\documentclass[psamsfonts,11pt]{amsart}
\usepackage[left=1in,right=1in,top=1.00in,bottom=1.00in]{geometry}
\usepackage{amssymb,mathrsfs}
\newtheorem{theorem}{Theorem}[section]
\newtheorem{lemma}[theorem]{Lemma}
\newtheorem{corollary}[theorem]{Corollary}
\newtheorem{proposition}[theorem]{Proposition}

\theoremstyle{definition}
\newtheorem{definition}[theorem]{Definition}
\newtheorem{example}[theorem]{Example}
\newtheorem{question}[theorem]{Question}
\newtheorem{remark}[theorem]{Remark}

\numberwithin{equation}{section}

\newcommand\N{\mathbb{N}}
\newcommand\R{\mathbb{R}}
\newcommand\T{\mathbb{T}}
\newcommand\Z{\mathbb{Z}}

\newcommand\cont{\mathfrak{c}}

\newcommand\fraction[2]{{#1}/{#2}}

\newcommand\grp[1]{\langle{#1}\rangle}

\title[Markov's potential density]{Hewitt-Marczewski-Pondiczery type theorem for abelian groups and Markov's potential density}

\author[D. Dikranjan]{Dikran Dikranjan}
\address
{Universit\`a di Udine, Dipartimento di Matematica e Informatica
\\
 via delle Scienze, 206 - 33100 Udine, Italy}
\email{dikran.dikranjan@dimi.uniud.it}
\thanks{The first named author was partially supported by SRA, grants P1-0292-0101 and J1-9643-0101.}

\author[D. Shakhmatov]{Dmitri Shakhmatov}
\address
{Graduate School of Science and Engineering,
Division of Mathematics, Physics and Earth Sciences\\
Ehime University, Matsuyama 790-8577, Japan}
\email{dmitri@dpc.ehime-u.ac.jp}
\thanks{The second named author was partially supported by the Grant-in-Aid for Scientific Research (C) No.~19540092 by the Japan Society for the Promotion of Science (JSPS)}



\begin{document}
\begin{abstract}
For an uncountable cardinal $\tau$ and a subset $S$ of an abelian group $G$, the following conditions are equivalent: 
\begin{itemize}
\item[(i)] $|\{ns:s\in S\}|\ge \tau$ for all integers $n\ge 1$;
\item[(ii)] there exists a group homomorphism $\pi:G\to \T^{2^\tau}$ such that $\pi(S)$ is dense in  $\T^{2^\tau}$. 
\end{itemize}
Moreover, if $|G|\le 2^{2^\tau}$, then the following item can be added to this list:
\begin{itemize}
\item[(iii)] there exists an isomorphism $\pi:G\to G'$ between $G$ and a subgroup $G'$ of $\T^{2^\tau}$ such that $\pi(S)$ is dense in  $\T^{2^\tau}$. 
\end{itemize}

We prove that the following conditions are equivalent for an uncountable subset $S$ of an abelian group $G$ that is either (almost) torsion-free or divisible:
 \begin{itemize}
  \item[(a)] $S$ is $\mathscr{T}$-dense in $G$ for some Hausdorff group topology $\mathscr{T}$ on $G$;
  \item[(b)] $S$ is $\mathscr{T}$-dense in some precompact Hausdorff group topology $\mathscr{T}$ on $G$;
  \item[(c)] $|\{ns:s\in S\}|\ge \min\left\{\tau:|G|\le 2^{2^\tau}\right\}$ for every integer $n\ge 1$.
\end{itemize}
This partially resolves a question of Markov going back to 1946.
 \end{abstract}

\maketitle

We use $\N$  and $\mathbb{P}$  to denote the set of  positive natural numbers and prime numbers, respectively. 
The symbol $\cont$ denotes the cardinality of the continuum, and $\omega$ denotes the first infinite cardinal. For an infinite cardinal $\kappa$, let 
$$
\log\kappa=\min\{\tau\ge\omega: \kappa\le 2^\tau\}.
$$

The group of integer numbers is denoted by $\Z$.

For a subset $X$ of an abelian group $G$, we use $\grp{X}$ to denote the subgroup of $G$ generated by $X$, and let $nX=\{nx:x\in X\}$ for every $n\in\Z$. As usual, 
$$
t(G)=\{g\in G: ng=0
\mbox{ for some }
n\in \N\}
$$
denotes the  {\em torsion part\/} of $G$. An abelian group $G$ is said to be:
\begin{itemize} 
\item[(i)]   {\em torsion-free\/} if $t(G)=\{0\}$;
\item[(ii)]  {\em torsion\/} if $t(G)=G$;
\item[(iii)] {\em bounded\/} if $nG=\{0\}$ for some $n\in \N$; 
\item[(iv)] {\em divisible\/} if $nG=G$ for every $n\in\N$. 
\end{itemize}

Recall that a group homomorphism $\varphi:G\to H$ from an abelian group $G$ to an abelian group $H$ is a {\em monomorphism\/} if $\ker \varphi=\{0\}$. 

A topological group is {\em precompact\/} if 
it is isomorphic to a subgroup of some compact Hausdorff group.
As usual, $w(X)$ denotes the {\em weight\/} of a topological space $X$.

\section{Introduction} 

The classical result of Hewitt-Marczewski-Pondiczery states: If $\tau$ is an infinite cardinal, $I$ is a set such that  $|I|\le 2^\tau$, and for every $i\in I$ a space $X_i$ has a dense subset of size $\le\tau$, then the product  $X=\prod\{X_i:i\in I\}$
also has a dense subset of size $\le\tau$ (see, for example, \cite[Theorem 2.3.15]{Eng}). In this paper we investigate the following ``algebraic version'' of this theorem. Let $\kappa$ be an infinite cardinal and $\T=\R/\Z$ 
the circle group. Given a fixed subset $S$ of an abelian group $G$, we attempt to find a group 
homomorphism $\pi:G\to \T^\kappa$ such that $\pi(S)$ becomes dense in $\T^\kappa$. Of particular interest is the special case when $\pi$ can be chosen to be a monomorphism, that is, when the group $G$ and the subgroup $\pi(G)$ of $\T^\kappa$ become isomorphic. Our choice of the target group is justified by the fact that every abelian group $G$ is isomorphic to a subgroup of $\T^\kappa$ for a suitable 
cardinal $\kappa$. To ensure a closer resemblance of the Hewitt-Marczewski-Pondiczery theorem, we pay special attention to the case $\kappa=2^\tau$ for some infinite cardinal
$\tau$ by addressing the following question: Given a subset $S$ of an abelian group $G$ such that $|S|\ge\tau$ and $|G|\le 2^{2^\tau}$, does there exist a monomorphism $\pi:G\to \T^{2^\tau}$ such that $\pi(S)$ becomes dense in $\T^{2^\tau}$? We completely resolve this problem in the case when $S$ is uncountable. Moreover, we provide a complete answer to this problem even for a countable set $S$ such that no multiple $nS$ of $S$ (for $n\in\N$) is contained in a finitely generated subgroup of $G$. The remaining case is being resolved by the authors in~\cite{DS_Kronecker}.
 
The origin of this setting can be traced back to the 1916 paper of  Weyl~\cite{W}.  We recall the classical Weyl's uniform distribution theorem: Given a faithfully indexed set $S=\{a_n:n\in\N\}\subseteq \Z$, the set of all $\alpha\in \T$ such that the set $S\alpha=\{a_n\alpha: n\in \N\}\subseteq \T$ is  uniformly distributed has full measure $1$. Since uniform distribution  implies density in $\T$, it follows that $S\alpha$ is dense in $\T$ for almost all $\alpha\in \T$. Every $\alpha\in \T$ determines uniquely a homomorphism $h_\alpha: \Z \to \T$ such that $h_\alpha(1) = \alpha$. Furthermore, $\alpha\in \T$ generates a dense subgroup $\grp{\alpha}$ of $\T$ iff  $\alpha$ is  non-torsion iff the  homomorphism $h_\alpha$ is a monomorphism. Hence, one can state (a consequence of) Weyl's theorem by simply saying that for every infinite subset $S$ of $\Z$, there exists a monomorphism $\pi: \Z \to \T$ such that $\pi(S)$ is dense in $\T$. Tkachenko and Yaschenko~\cite{TY} consider homomorphisms $\pi:G\to \T^\omega$ such that $\pi(S)$ is dense in $\T^\omega$, for a certain class of groups $G$. They use such homomorphisms as a technical tool in addressing the problem suggested first in 1946 by Markov~\cite{Mar}.

According to Markov~\cite{Mar}, a subset $S$ of a group $G$ is called {\em potentially dense\/} (in $G$) provided that $G$ admits some Hausdorff  group topology $\mathscr{T}$ such that $S$ is dense in $(G,\mathscr{T})$. The last section of~\cite{Mar}  is exclusively dedicated to the following problem: Which subsets of a group $G$ are potentially dense in $G$? Markov  proved that every infinite subset of $\Z$ is potentially dense in $\Z$~\cite{Mar}. This was strengthened in \cite[Lemma 5.2]{DT0} by showing that every infinite subset of $\Z$ is dense in some precompact metric group topology on $\Z$. (Apparently,  the authors of~\cite{Mar} and~\cite{DT0} were unaware that both these results easily follow from  Weyl's uniform distribution theorem.) Further progress was obtained by Tkachenko and Yaschenko~\cite{TY} who proved  the following theorem: If an abelian group $G$ of size at most $\cont$  is either almost  torsion-free or has exponent $p$ for some prime $p$, then every infinite subset of $G$ is potentially dense in $G$. (According to~\cite{TY}, an abelian group $G$  is {\em almost torsion-free\/} if $r_p(G)$ is finite for every prime $p$.)   

In~\cite{DS_Kronecker}, the authors resolved Markov's problem for {\sl countable\/} sets: A countable subset $S$ of an abelian group $G$ is potentially dense in $G$ if and only if $|G|\le 2^\cont$ and $S$ is Zariski dense in $G$. Recall that a subset $S$ of an abelian group $G$ is said to be {\em Zariski dense\/} in $G$  provided that, if $k\in \N$, $g_1, g_2, \ldots, g_k\in G$, $n_1, n_2, \ldots, n_k\in \N$ and  each $s\in S$ satisfies the equation $n_is=g_i$ for some $i= 1,2,\ldots, k$ (depending on $s$),  then every $g\in G$ also satisfies some equation $n_jg=g_j$, for a suitable $j= 1,2,\ldots, k$ (\cite{DS_Abelian_MZ}; see also \cite[Section 5]{DS_OPIT}). 

In this manuscript we investigate the remaining case of {\sl uncountable\/} sets. In particular, we obtain a new sufficient condition  that guarantees that a  subset $S$ of an abelian group $G$  is potentially dense in $G$. Moreover, when this condition is satisfied, we prove that the topology $\mathscr{T}$ on $G$ such that $S$ is $\mathscr{T}$-dense in $G$ can be chosen to be precompact. When $S$ is uncountable and  $G$ belongs to a wide class of abelian groups (for example, almost torsion-free or divisible groups), our sufficient condition turns out to be also necessary for potential density of $S$ in $G$.  At last but not least, our sufficient condition is rather powerful in the countable case as well, becuase the only case that is {\em not\/} covered by it is when $nS$, for a suitable $n\in\N$, is contained in a finitely generated subgroup of $G$. Therefore, it is only this special case that still requires the  substantially more sophisticated techniques from~\cite{DS_Kronecker} to prove potential density (in some precompact group topology).

\section{Sending a given subset of an abelian group densely in $\T^\kappa$}

\begin{definition}
Let $\tau$ be an infinite cardinal. We say that a subset $S$ of an abelian group $G$ is {\em $\tau$-wide\/} if $nS\setminus \grp{S'}\not=\emptyset$ for every $n\in \N$ and each $S'\subseteq S$ with $|S'|<\tau$.
\end{definition}

Our next proposition collects four simple facts that clarify the above definition and facilitate future references.

\begin{proposition}
\label{tau-wide:remark}
Let $S$ be a subset of an abelian group $G$.
\begin{itemize}
\item[(i)] If $S$ is uncountable, then $S$ is $\tau$-wide if and only if $\min\{|nS|:n\in\N\}\ge\tau$. 
\item[(ii)] If $S$ contains an infinite independent subset $S'$, then $S$ is $|S'|$-wide. 
\item[(iii)] If $G$ is torsion, then $S$ is $\omega$-wide if and only if $nS$ is infinite for every $n\in\N$ (that is, $S$ is 0-almost torsion in the sense of ~\cite{DS-Forcing}). Therefore, $G$ contains an $\omega$-wide set $S$ if and only if $G$ is unbounded (and in such a case the $\omega$-wide sets coincide with the Zariski dense sets). 
\item[(iv)] If $S$ is not $\omega$-wide, then $nS$ is contained in a finitely generated subgroup of $G$, for a suitable $n\in\N$.
\end{itemize}
\end{proposition}

\begin{theorem}
\label{HMP}
Let $\tau$ be an infinite  cardinal and $S$ a $\tau$-wide subset of an abelian group $G$.
Then there exist a subgroup $H$ of $G$ and a monomorphism $\varphi:H\to \T^{2^\tau}$ such that $|H|=\tau$ and $\varphi(H\cap S)$ is dense in $\T^{2^\tau}$.
\end{theorem}

The proof of this theorem is postponed until Section~\ref{Proof:section}.

\begin{corollary}
\label{main:corollary}
Let $\tau$ be an infinite  cardinal and $S$ a $\tau$-wide subset of an abelian group $G$. Then:
\begin{itemize}
\item[(i)] there exists a group homomorphism $\pi:G\to \T^{2^\tau}$ such that $\pi(S)$ is dense in  $\T^{2^\tau}$;
\item[(ii)] if one additionally assumes that $|G|\le 2^{2^\tau}$, then $\pi$ from the item (i) can be chosen to be a monomorphism.
\end{itemize}
\end{corollary}
\begin{proof} Let $H$ and $\varphi$ be as in the conclusion of Theorem~\ref{HMP}.  The proof now branches into two cases, depending on which item of our corollary holds.

(i)  Since $\T^{2^\tau}$ is a divisible group, there exists a group homomorphism  $\pi:G\to \T^{2^\tau}$ extending $\varphi$. 

(ii)  Since $|H|= \tau < 2^{2^\tau}$ and $r_p(G)\le |G|\le 2^{2^\tau}=\left|\T^{2^\tau}\right|=r_p\left(\T^{2^\tau}\right)$ for every $p\in\mathbb{P}\cup\{0\}$, we can extend $\varphi$ to a monomorphism  $\pi:G\to \T^{2^\tau}$, see~\cite[Lemma 3.17]{DS-Forcing}. 

Returning back to the common  part of the proof,  note that $\varphi(H\cap S)\subseteq \pi(S)$ and $\varphi(H\cap S)$ is dense in $\T^{2^\tau}$, so $\pi(S)$ must be dense in $\T^{2^\tau}$ as well.
\end{proof}

\begin{theorem}
\label{dense:mappings:into:Tkappa}
Let $\kappa$ be a cardinal such that $\kappa>\cont$, and let $\tau=\log\kappa$. For a subset $S$ of an abelian group $G$, the following conditions are equivalent:
\begin{itemize}
\item[(i)] there exists a group homomorphism $\varpi:G\to \T^\kappa$ such that $\varpi(S)$ is dense in  $\T^\kappa$;
\item[(ii)] $\tau\le \min\{|nS|:n\in\N\}$;
\item[(iii)] $S$ is $\tau$-wide.
\end{itemize}
\end{theorem}

\begin{proof} (i)$\to$(ii) 
Let $n\in \N$.  Since $\pi(S)$ is  dense in $\T^\kappa$, $w(n\T^\kappa)\le 2^{|nS|}$ by Lemma \ref{necessary:condition}(ii) below. Since $\T^\kappa$ is divisible, $n\T^\kappa=\T^\kappa$, and so $\kappa=w(\T^\kappa)\le 2^{|nS|}$, which yields $\tau=\log\kappa\le |nS|$. This proves (ii).

(ii)$\to$(iii)  Since $\kappa>\cont$, $\tau$ is uncountable. Applying (ii) with $n=1$, we conclude that $S$ is uncountable. Then $S$ is $\tau$-wide by (ii) and   Proposition~\ref{tau-wide:remark}(i).

(iii)$\to$(i) Let $\pi$ be as in the conclusion of  Corollary~\ref{main:corollary}(i). From $\tau=\log \kappa$ it follows that $\kappa\le 2^\tau$. Let $\psi:\T^{2^\tau}\to\T^\kappa$ be the projection on the first $\kappa$ coordinates.  Since $\psi$ is a homomorphism, so is $\varpi=\psi\circ \pi: G\to \T^\kappa$. Since $\psi$ is continuous and $\pi(S)$ is dense in $\T^{2^\tau}$, the set $\psi(\pi(S))=\varpi(S)$ is dense in $\psi(\T^{2^\tau})=\T^\kappa$.
\end{proof}

\begin{theorem}
\label{dense:monomorphisms:into:Tkappa} Let $\kappa$ be  a cardinal such that $\kappa>\cont$, and let $\tau=\log\kappa$. For a subset $S$ of an abelian group $G$, the following conditions are equivalent:
\begin{itemize}
\item[(i)] there exists a monomorphism $\varpi:G\to \T^\kappa$ such that $\varpi(S)$ is dense in  $\T^\kappa$;
\item[(ii)] $|G|\le 2^\kappa$ and $\tau\le \min\{|nS|:n\in\N\}$;
\item[(iii)] $|G|\le 2^\kappa$ and $S$ is $\tau$-wide.
\end{itemize}
\end{theorem}

\begin{proof} (i)$\to$(ii) Clearly, $|G|=|\varpi(G)|\le |\T^\kappa|=2^\kappa$. The other inequality in item (ii)  follows from the implication (i)$\to$(ii) of Theorem~\ref{dense:mappings:into:Tkappa}. 

(ii)$\to$(iii) follows from the implication (ii)$\to$(iii) of Theorem~\ref{dense:mappings:into:Tkappa}.

(iii)$\to$(i) Let $H$ and $\varphi$ be as in the conclusion of Theorem~\ref{HMP}. From $\tau=\log \kappa$ it follows that $\tau\le\kappa\le 2^\tau$.

For every $h\in H\setminus \{0\}$, there exists $\xi_h\in 2^\tau$ such that $\varphi(h)(\xi_h)\not=0$. Define  $\Xi=\{\xi_h\in 2^\tau :h\in H\setminus \{0\}\}\subseteq 2^\tau$. 
Let $\mu: \T^{2^\tau}\to \T^\Xi$ be the projection. Then $\mu\restriction_{\varphi(H)}:\varphi(H)\to \T^\Xi$ is a monomorphism. Since  $|\Xi|\le |H\setminus \{0\}|\le |H|=\tau\le\kappa$, it follows that there exists a homomorphism $\psi:\T^{2^\tau}\to\T^\kappa$ such that $\chi=\psi\circ \varphi: H\to \T^\kappa$  is a monomorphism.
Since $\psi$ is continuous and $\varphi(H\cap S)$ is dense in $\T^{2^\tau}$, the set $\psi(\varphi(H\cap S))=\chi(H\cap S)$ is dense in $\psi(\T^{2^\tau})=\T^\kappa$.

Since $|H|= \tau \le\kappa< 2^{\kappa}$ and $r_p(G)\le |G|\le 2^{\kappa}=|\T^{\kappa}|=r_p(\T^{\kappa})$ for every $p\in\mathbb{P}\cup\{0\}$, we can extend $\chi$ to a monomorphism  $\varpi:G\to \T^{\kappa}$; see \cite[Lemma 3.17]{DS-Forcing}. Since $\chi(H\cap S)\subseteq \varpi(S)$  and $\chi(H\cap S)$ is dense in $\T^{\kappa}$, the set $\varpi(S)$ must be dense in $\T^{\kappa}$ as well.
\end{proof}

The counter-part of Theorems~\ref{dense:mappings:into:Tkappa} and~\ref{dense:monomorphisms:into:Tkappa} for $\kappa\le\cont$ is proved in~\cite{DS_Kronecker}.

\section{Markov's potential density}\label{potential:density:section}

We start this section with a simple necessary condition for potential density.

\begin{lemma}
\label{necessary:condition}
Let $S$ be a dense subset of a Hausdorff group $G$. Then:
\begin{itemize}
\item[(i)]  for every  $n\in\N$, the set $nS$ is dense in the subgroup $nG$ of $G$;
\item[(ii)] $w(nG)\le 2^{|nS|}$ and $|nG|\le 2^{2^{|nS|}}$ for each $n\in\N$. 
\end{itemize}
\end{lemma}

\begin{proof} (i)  The map $g\mapsto ng$ that sends $G$ to the subgroup $nG$ of $G$ is continuous. Since $S$ is dense in $nG$, $nS$ must be dense in $nG$.

(ii) Let $n\in \N$. From (i) we conclude that $nS$ is a dense subset of  the Tychonoff space $nG$. Therefore, $w(nG)\le 2^{|nS|}$ by \cite[Theorem 1.5.7]{Eng}  and $|nG|\le 2^{2^{|nS|}}$ by \cite[Theorem 1.5.3]{Eng}. 
\end{proof}

\begin{corollary}
\label{necessary:condition:corollary} If $S$ is a potentially dense subset of a group $G$, then 
\begin{equation}
\label{necessary:condition:equation}
\mbox{$\log \log |nG| \leq |nS|$ \ for all $n\in\N$.}
\end{equation}
\end{corollary}

Our next theorem provides a general sufficient condition for potential density that also allows it to be realized by some precompact group topology.

\begin{theorem}
\label{sufficient:condition:for:potential:density}
Let $\tau$ be an infinite  cardinal and $S$ a $\tau$-wide subset of an abelian group $G$ such that $|G|\le 2^{2^\tau}$.
Then there exist a precompact Hausdorff group topology $\mathscr{T}$ on $G$  such that $S$ is dense in $(G,\mathscr{T})$. 
\end{theorem}
\begin{proof}
Let $\pi:G\to \T^{2^\tau}$ be a monomorphism from Corollary~\ref{main:corollary}(ii), and let $\mathscr{T}'$ be the topology $\pi(G)$ inherits from $\T^{2^\tau}$. Then $(G,\mathscr{T}')$ is a precompact group. Since $\pi(S)$ is dense in $\T^{2^\tau}$, we conclude that $\pi(S)$ is $\mathscr{T}'$-dense in $\pi(G)$. Since $\pi$ is an isomorphism between $G$ and $\pi(G)$, $\mathscr{T}=\{\pi^{-1}(U):U\in \mathscr{T}'\}$ is the required topology.
\end{proof}

The characterization of the countable potentially dense  subsets of an abelian group $G$ with $|G|\leq 2^\cont$ can be found in~\cite{DS_Kronecker}.  Our next corollary shows that Theorem~\ref{sufficient:condition:for:potential:density} is sufficiently useful in obtaining some particular cases of that characterization.

\begin{corollary}
Let $G$ be an unbounded torsion abelian group with $|G|\leq 2^\cont$.  For a countable subset $S$ of $G$,  the following conditions are equivalent:
\begin{itemize}
\item[(i)] $S$ is potentially dense in $G$;
\item[(ii)] there exist a precompact Hausdorff group topology $\mathscr{T}$ on $G$  such that $S$ is dense in $(G,\mathscr{T})$;
\item[(iii)] $S$ is $\omega$-wide.
 \end{itemize}
\end{corollary}

\begin{proof} The implication (iii)$\to$(ii) follows from Theorem~\ref {sufficient:condition:for:potential:density}, while the 
implication (ii)$\to$(i) is trivial. According to  Proposition~\ref{tau-wide:remark}(iii), to prove the implication (i)$\to$(iii), it suffices to check that  $nS$ is infinite for every $n \in \N$. Assume that $nS$ is finite for some $n \in \N$. By Lemma~\ref{necessary:condition}(i), this yields that $nG$ is finite as well. Consequently, $G$ is bounded,  a contradiction. \end{proof}
In connection with the last corollary, it is worth mentioning that the unboundedness of  a torsion abelian group $G$ is a necessary condition for the existence of an $\omega$-wide subset of $G$; see Proposition~\ref{tau-wide:remark}(iii).

\begin{corollary}
\label{nG:G}
Let  $G$ be an abelian group such that $|nG|=|G|$ for every $n\in \N$. For an uncountable subset $S$ of $G$, the following conditions are equivalent:
\begin{itemize}
  \item[(i)]   $S$ is potentially dense in $G$;
  \item[(ii)]  $S$ is $\mathscr{T}$-dense in some precompact group topology $\mathscr{T}$ on $G$;
  \item[(iii)] $\log\log|nG|\le |nS|$ for every $n\in \N$;
  \item[(iv)] $\log\log|G|\le \min\{|nS|:n \in \N\}$.
\end{itemize}
\end{corollary}
\begin{proof}
The implication (ii)$\to$(i) is obvious, the  implication (i)$\to$(iii) is proved in Corollary~\ref{necessary:condition:corollary},
and the implication (iii)$\to$(iv) holds due to our assumption on $G$. It remains only to check the implication (iv)$\to$(ii). From (iv) and Proposition~\ref{tau-wide:remark}(i) we conclude that $S$  is $\tau$-wide for $\tau=\log\log|G|$. Since $|G|\le 2^{2^\tau}$, Theorem~\ref{sufficient:condition:for:potential:density} applies.
\end{proof}

\begin{corollary}
\label{unified:corollary}
Let $G$ be an abelian group satisfying one of the following conditions:
\begin{itemize}
\item[(a)] $G$ is divisible;
\item[(b)] $|t(G)|<|G|$;
\item[(c)] $|t(G)|\le\omega$;
\item[(d)] $G$ is almost torsion-free.
\end{itemize}
For every uncountable subset $S$ of $G$, conditions (i)--(iv) of Corollary~\ref{nG:G} are equivalent.
\end{corollary}

\begin{proof}
Let $S$ be an uncountable subset of $G$. In particular, $|G|\ge |S|>\omega$.
It suffices to show that $G$ satisfies the assumption  of Corollary~\ref{nG:G}.  Fix $n\in\N$. If (a) holds, then $|nG|=|G|$ trivially holds,  as the homomorphism
$\eta_n: G\to G$ defined by $\eta_n(g)=ng$ for $g\in G$, is surjective. If (b) holds, then from  $\ker\eta_n\subseteq t(G)$ and $|t(G)|< |G|$ it follows that $|nG|=|G|$. If (c) holds, then $t(G)\le \omega < |G|$,  and so (c)
is a particular case of (b).  Finally, note that every almost torsion-free group satisfies (c).
\end{proof}

\begin{remark}\label{last:remark}
For an uncountable subset $S$ of an abelian group $G$ satisfying items (c) or (d) of Corollary~\ref{unified:corollary}, the following simplified condition can be added to the list of equivalent items (i)--(iv) of Corollary
\ref{nG:G}:
\begin{itemize}
\item[(v)] $\log\log|G|\le |S|$.
\end{itemize}
Indeed,  since $t(G)$ is at most countable and $S$ is uncountable, we must have $|nS|=|S|$ for every $n\in\N$. This establishes the implication (v)$\to$(iv). The reverse implication (iv)$\to$(v) is trivial. 
\end{remark}

Let $S$ be an uncountable subset of an abelian group $G$. Corollaries~\ref{nG:G} and~\ref{unified:corollary}
show that  (\ref{necessary:condition:equation})  is not only a necessary, but also a sufficient condition  for potential density of 
$S$  in $G$ when $G$ belongs to a wide class of groups including divisible and (almost) torsion-free groups. 
According to Remark~\ref{last:remark}, the weaker (and much simpler) condition $\log \log |G| \leq |S|$, obtained by taking $n=1$ in (\ref{necessary:condition:equation}),  is  also sufficient for potential density of $S$ in $G$ when the torsion part $t(G)$  of $G$ is at most countable. Our next example demonstrates that  this simpler condition (v) is no longer sufficient for  potential density of $S$ in $G$
when one weakens the assumption $|t(G)|\le \omega$ to $|t(G)|<|G|$. (Note that $|G|\ge|S|>\omega$.)

\begin{example}
\label{potentially:dense:example}
 Let $\Z(2)=\Z/2\Z$ be the abelian group with two elements. Define $G=\Z(2)^\omega \times \Z^{2^\cont}$ and $S = \Z(2)^\omega \times C$,  where $C$ is any infinite cyclic subgroup of $ \Z^{2^\cont}$. 
Then $\log \log |G|= \cont = |S|$ and  $|t(G)|=|\Z(2)^\omega \times\{0\}|=\cont<|G|$. 
Since $\log \log |2G| > \omega= |2S|$, the set $S$ cannot be
potentially dense in $G$ by Corollary~\ref{necessary:condition:corollary}. 
\end{example}

Observe that that  a group $G$ with the properties from Example~\ref{potentially:dense:example} must necessarily have  
an uncountable torsion part $t(G)$; see Remark~\ref{last:remark}.
 
Example~\ref{potentially:dense:example} provides a negative answer to \cite[Question 45]{DS_OPIT}.

\section{Technical lemmas}

We call a connected open subset $V$ of $\T$ an  {\em open arc\/}, and we use $l(V)$ to denote the length of $V$.

\begin{lemma}
\label{torus:lemma}
Suppose that $V$ is an open arc in $\T$, $z,z'\in \T$, $m,n\in\N$, $1\le n<m$ and $\fraction{2}{m}<l(V)$. Then there exists  $y\in V$ such that $my=z$ and $ny\not=z'$. 
\end{lemma}
\begin{proof}
Since $\fraction{2}{m}<l(V)$, the arc $V$ contains two solutions  $y_0$ and $y_1$ of the equation $my=z$ with  
\begin{equation}\label{(*)}
l(C(y_0,y_1))=\fraction{1}{m},
\end{equation}
where $C(y_0,y_1)$ is the shortest arc in $\T$ connecting $y_0$ and $y_1$.

Suppose that $ny_0=ny_{1}=z'$. Then $n(y_0-y_{1})=0$, and so $l(C(y_0,y_1))\ge \fraction{1}{n}$. Together with (\ref{(*)}), this gives $m\le n$, a contradiction. Therefore,  $ny_{j}\not=z'$ for some $j=0,1$, and so we can take that $y_{j}$ as our $y$.
\end{proof}
 
If $G$ and $H$ are groups, then $G\cong H$ means that $G$ and $H$ are isomorphic.

\begin{lemma}
\label{second:lemma}
Let $\tau$ and $\kappa$ be infinite cardinals such that $\tau\le\kappa$. Let $K$ be a subgroup of $\T^\kappa$  such that $|K|\le\tau$. For every $\gamma<\kappa$, let $V_\gamma$ be an open arc in $\T$ such that the set $\{l(V_\gamma):\gamma<\kappa \}$ has a positive lower bound. Then there exists $f\in \prod\{V_\gamma:\gamma<\kappa\}$ having the following properties:
\begin{itemize}
\item[(i)]  $\grp{f}\cong\Z$;
\item[(ii)] $\grp{f}\cap K=\{0\}$.
\end{itemize}
\end{lemma}

\begin{proof} Fix $k\in\N$ such that  $\fraction{2}{k}<l(V_\gamma)$ for every $\gamma<\kappa$. 
Since $|K|\le\tau$, we can choose an enumeration  $K\times\N=\{(h_\alpha,n_\alpha):\alpha<\tau\}$ of the set $K\times \N$.  By transfinite recursion on $\alpha<\tau$ we will select $\gamma_\alpha<\kappa$ and $y_{\gamma_\alpha}\in\T$ satisfying the following properties:
\begin{itemize} 
\item[(i$_\alpha$)]   $\gamma_\alpha\not\in\{\gamma_\beta:\beta<\alpha\}$,
\item[(ii$_\alpha$)]  $y_{\gamma_\alpha}\in V_{\gamma_\alpha}$,
\item[(iii$_\alpha$)] $n_\alpha y_{\gamma_\alpha}\not=h_\alpha(\gamma_\alpha)$.
\end{itemize}

\smallskip
{\sl Basis of recursion\/}. Select $\gamma_0<\kappa$ arbitrarily. Apply Lemma~\ref{torus:lemma} to $V=V_{\gamma_0}$, $z=0$, $z'=h_0(\gamma_0)$, $n=n_0$ and $m=n_0 (k+1)$ to choose $y_{\gamma_0}\in\T$ satisfying (ii$_0$) and (iii$_0$).  Condition (i$_0$) is vacuous.

\smallskip
{\sl Recursive step\/}. Let  $\alpha<\tau$,  and suppose that $\gamma_\beta<\kappa$ and $y_{\gamma_\beta}\in\T$ satisfying (i$_\beta$)--(iii$_\beta$) 
 have already been selected for all $\beta<\alpha$. We will now choose $\gamma_\alpha<\kappa$ and $y_{\gamma_\alpha}\in\T$ satisfying  (i$_\alpha$)--(iii$_\alpha$).  
 Since $|\alpha|<\tau\le\kappa$, we can choose $\gamma_\alpha<\kappa$ satisfying (i$_\alpha$). Now we  apply Lemma~\ref{torus:lemma} to $V=V_{\gamma_\alpha}$, 
$z=0$, $z'=f_\alpha(\gamma_\alpha)$, $n=n_\alpha$ and $m=n_\alpha(k+1)$ to choose $y_{\gamma_\alpha}\in\T$ satisfying (ii$_\alpha$), and (iii$_\alpha$). 

\smallskip
The recursion being complete,  choose $y_\gamma\in V_\gamma$ arbitrarily for every $\gamma\in\kappa\setminus\{\gamma_\alpha:\alpha<\tau\}$. Define $f\in\T^\kappa$ by $f(\gamma)=y_\gamma$ for each $\gamma< \kappa$. Then $f\in \prod\{V_\gamma:\gamma<\kappa\}$. We claim that 
\begin{equation}
\label{not:in:the:subgroup}
nf\not\in K\ \ 
\mbox{ for every}\ \ n\in\N.
\end{equation}
Indeed, let $n\in\N$ and $h\in K$ be arbitrary. Then $(h,n)\in K\times \N$,  and so  $(h,n)=(h_\alpha,n_\alpha)$ for some $\alpha<\kappa$. Now  $nf(\gamma_\alpha)=n_\alpha y_{\gamma_\alpha}\not=h_\alpha(\gamma_\alpha)=h(\gamma_\alpha)$  by (iii$_\alpha$).  Thus, $nf\not=h$. Since $h\in K$ was arbitrary, this proves \eqref{not:in:the:subgroup}. From \eqref{not:in:the:subgroup} we immediately get both (i) and (ii).
\end{proof}

\begin{lemma}
\label{first:lemma}
Let $\tau$ and $\kappa$ be infinite cardinals such that $\tau\le\kappa$. Let $K$ be a subgroup of $\T^\kappa$ such that $|K|\le\tau$. Assume that $f'\in K$, $m\in\N$ and $m\ge 2$.  For every $\gamma<\kappa$, let $V_\gamma$ be an open arc in $\T$ such that $\fraction{2}{m}<l(V_\gamma)$. Then there exists  $f\in \prod\{V_\gamma:\gamma<\kappa\}$ satisfying the following properties:
\begin{itemize}
\item[(i)]   $mf=f'$;
\item[(ii)]  $nf\not\in K$ for all $n\in\N$ with $1\le n<m$.
\end{itemize}
\end{lemma}
\begin{proof}
Since $|K|\le\tau$, we can choose an enumeration  $K\times \{1,2,\dots,m-1\}=\{(h_\alpha,n_\alpha):\alpha<\tau\}$ of the set  $K\times \{1,2,\dots,m-1\}$. By transfinite recursion on $\alpha<\tau$ we will select $\gamma_\alpha<\kappa$ and $y_{\gamma_\alpha}\in\T$ with the following properties:
\begin{itemize}
\item[(i$_\alpha$)] $\gamma_\alpha\not\in\{\gamma_\beta:\beta<\alpha\}$,
\item[(ii$_\alpha$)] $y_{\gamma_\alpha}\in V_{\gamma_\alpha}$,
\item[(iii$_\alpha$)] $my_{\gamma_\alpha}=f'(\gamma_\alpha)$,
\item[(iv$_\alpha$)] $n_\alpha y_{\gamma_\alpha}\not=h_\alpha(\gamma_\alpha)$.
\end{itemize}

\smallskip
{\sl Basis of recursion\/}. Select $\gamma_0<\kappa$ arbitrarily, and apply Lemma~\ref{torus:lemma} to $V=V_{\gamma_0}$, $z=f'(\gamma_0)$, $z'=h_0(\gamma_0)$, $n=n_0$ and $m$ to choose $y_{\gamma_0}\in\T$ satisfying (ii$_0$), (iii$_0$) and (iv$_0$). Condition (i$_0$) is vacuous.

\smallskip
{\sl Recursive step\/}. Let $\alpha<\tau$, and suppose that $\gamma_\beta<\kappa$ and $y_{\gamma_\beta}\in\T$ satisfying (i$_\beta$)--(iv$_\beta$) have already been selected for all $\beta<\alpha$. We will now choose $\gamma_\alpha<\kappa$ and $y_{\gamma_\alpha}\in\T$ satisfying (i$_\alpha$)--(iv$_\alpha$). Since $|\alpha|<\tau\le\kappa$, we can choose $\gamma_\alpha< \kappa$ satisfying (i$_\alpha$). Now we  apply Lemma~\ref{torus:lemma} to $V=V_{\gamma_\alpha}$, 
$z=f'(\gamma_\alpha)$, $z'=h_\alpha(\gamma_\alpha)$, $n=n_\alpha$ and $m$ to choose $y_{\gamma_\alpha}\in\T$ satisfying (ii$_\alpha$), (iii$_\alpha$) and (iv$_\alpha$).

\smallskip
The recursion being complete, for every $\gamma\in \kappa\setminus\{\gamma_\alpha:\alpha<\tau\}$, apply Lemma~\ref{torus:lemma} to $V=V_{\gamma}$, 
$z=f'(\gamma)$, $z'=0$, $n=1$ and $m$ to choose $y_\gamma\in V_\gamma$ such that $my_{\gamma}=f'(\gamma)$.

Define $f\in\T^\kappa$ by $f(\gamma)=y_\gamma$ for every $\gamma< \kappa$. Then  $f\in \prod\{V_\gamma:\gamma<\kappa\}$ and (i) is satisfied. To prove (ii), choose $n\in\N$ such that $1\le n<m$. Let $h\in K$ be arbitrary. Then $(h,n)\in K\times \{1,2,\dots,m-1\}$, and so $(h,n)=(h_\alpha,n_\alpha)$ for some $\alpha<\kappa$. Now  $nf(\gamma_\alpha)=n_\alpha y_{\gamma_\alpha}\not=h_\alpha(\gamma_\alpha)=h(\gamma_\alpha)$ by (iv$_\alpha$). Therefore, $nf\not=h$. Since $h\in K$ was arbitrary, this proves $nf\not\in K$. 
\end{proof}

\begin{lemma} \label{Lemma:extension} 
Let $G$ and $G^*$ be abelian groups, and let $K$ and $K^*$ be subgroups of $G$ and $ G^*$, respectively. Suppose also that $x \in G$, $x^*\in G^*$, 
$m\in\N$, $m\ge 2$, and $\psi:K\to K^*$ is an isomorphism satisfying the following properties: 
\begin{itemize}
\item[(a)] $mx\in K$ and $mx^*\in K^*$;
\item[(b)] $nx\not\in K$ and $nx^*\not\in K^*$ whenever $n\in\N$ and $1\le n<m$;
\item[(c)] $\psi(mx)=mx^*$.
\end{itemize}
Then there exists a unique isomorphism $\varphi: K + \grp{x}\to K^*+ \grp{x^*} $ extending $\psi$ such that  $\varphi(x) = x^*$.
\end{lemma}

\begin{proof} Define $\varphi$ by 
\begin{equation}
\label{varphi:prime:equation}
\varphi(h+kx)=\psi(h)+kx^*
\mbox{ for }
h\in K
\mbox{ and }
k\in \Z.
\end{equation}

To check that this definition is correct, suppose that 
\begin{equation}\label{(**)}
h+kx=h'+k'x,
\end{equation} 
with   $h, h'\in K$ and $k, k'\in \Z$. Hence, $(k'-k)x=h-h'\in K$, so by (a)  and (b), one has $k'-k=lm$ for  some $l\in \Z$, which yields $h-h'=lmx$. This gives
$$
\psi(h)-\psi(h')=\psi(h-h')=\psi(lmx)=l\psi(mx)=lmx^*=(k'-k)x^*=k'x^*-kx^*
$$
by (c), and so $\psi(h) +kx^*=\psi(h') +k'x^*$. Comparing this  with  (\ref{(**)}), we conclude that  \eqref{varphi:prime:equation} correctly defines a homomorphism $\varphi$. From \eqref{varphi:prime:equation} we get $\varphi(x) = x^*$ and $\varphi\restriction_K=\psi$.  Moreover,  $\varphi$ is surjective and unique with these properties. 

To prove that $\varphi$ is a monomorphism, assume that $\varphi(h+kx) = 0$ for some $h\in K$ and $k\in \Z$. Then $\psi(h) + kx^*=0$, so  $kx^*\in K^*$. Consequently, $m$ divides $k$ by (a) and (b). Then $kx\in K$, and so $ h+kx \in K$ as well. Therefore, $\varphi(h+kx) =\psi(h+kx) =0$ yields $ h+kx=0$, as $\psi$ is a monomorphism. 
\end{proof}

\section{Proof of Theorem \ref{HMP}}

\label{Proof:section}
Fix a countable base $\mathscr{V}$ for the topology of $\T$ consisting of open arcs of $\T$ such that $\T\in\mathscr{V}$.
Consider the Tychonoff product topology on $2^\tau$, and let $\mathscr{B}$ be the canonical base for $2^\tau$ consisting of non-empty clopen subsets of $2^\tau$ such that $|\mathscr{B}|=\tau$. Let
$$
\mathbb{U}=\left\{\mathscr{U}\in[\mathscr{B}]^{<\omega}: \mathscr{U}
\mbox{ is a cover of }
2^\tau
\mbox{ by pairwise disjoint sets}
\right\}.
$$
For $\xi\in 2^\tau$ and $\mathscr{U}\in \mathbb{U}$, let $U_{\xi,\mathscr{U}}\in\mathscr{U}$ denote the unique $U\in \mathscr{U}$ such that $\xi\in U$. Define 
$$
\mathbb{E}=\{(\mathscr{U}, v): \mathscr{U}\in\mathbb{U} \mbox{ and }
v:\mathscr{U}\to \mathscr{V}
\mbox{ is a function}
\}.
$$
 For $(\mathscr{U}, v)\in \mathbb{E}$, let 
$$
F(\mathscr{U}, v)=\left\{f\in\T^{2^\tau}: f(\xi)\in v(U_{\xi,\mathscr{U}}) \mbox{ for all }\xi\in 2^\tau\right\}=\prod \left\{v(U_{\xi,\mathscr{U}}): \xi\in 2^\tau\right\}.
$$

Clearly, $|\mathbb{E}|=\tau$, so we can fix an enumeration $\mathbb{E}=\{(\mathscr{U}_\alpha, v_\alpha):\alpha<\tau\}$ of $\mathbb{E}$
such that $\mathscr{U}_0=\left\{2^\tau\right\}$ and $v_0\left(2^\tau\right)=\T$. For each $\alpha<\tau$, choose $n_\alpha\in\N$ such that 
\begin{equation}
\label{n:alpha}
\fraction{2}{n_\alpha}<\min\{l(v(U)):U\in\mathscr{U}\}.
\end{equation}

By transfinite recursion on $\alpha<\tau$ we will choose an element $x_\alpha\in S$ and define a map $\varphi_\alpha:H_\alpha=\grp{\{x_\beta:\beta\le\alpha\}}\to\T^{2^\tau}$ satisfying the following conditions:
\begin{itemize}
 \item[(i$_\alpha$)]   $\varphi_\alpha(x_\alpha)\in F(\mathscr{U}_\alpha, v_\alpha)$,
 \item[(ii$_\alpha$)]  $\varphi_\alpha$ is a monomorphism,
 \item[(iii$_\alpha$)] $\varphi_\alpha\restriction_{H_\beta}=\varphi_\beta$ for all $\beta<\alpha$.
\end{itemize}

\smallskip
{\sl Basis of recursion\/}. Pick $x_0\in S$ arbitrarily, and let $\varphi_0:\grp{x_0}\to \T^{2^\tau}=F(\mathscr{U}_0,v_0)$ be an arbitrary monomorphism. 
Then conditions (i$_0$) and  (ii$_0$) are satisfied,  while the condition (iii$_0$)  is vacuous.

\smallskip
{\sl Recursive step\/}. Let $\alpha<\tau$, and suppose that $x_\beta\in S$ and a map $\varphi_\beta:H_\beta\to\T^{2^\tau}$ satisfying (i$_\beta$), (ii$_\beta$) and (iii$_\beta$) have already been constructed for every $\beta<\alpha$. We are going to define $x_\alpha\in S$ and a map $\varphi_\alpha:H_\alpha\to\T^{2^\tau}$ satisfying  (i$_\alpha$), (ii$_\alpha$) and  (iii$_\alpha$).

Define 
$$
H'_\alpha=\grp{\{x_\beta:\beta<\alpha\}}=\bigcup_{\beta<\alpha} H_\beta.
$$
Since (ii$_\beta$) and (iii$_\beta$) hold for every $\beta<\alpha$,  
$$
\varphi'_\alpha=\bigcup_{\beta<\alpha} \varphi_\beta:H'_\alpha\to \T^{2^\tau}$$
 is a monomorphism.
Since $\{x_\beta:\beta<\alpha\}\subseteq S$,  $|\alpha|<\tau$ and $S$ is $\tau$-wide, one has
$
n_\alpha! S\setminus \grp{\{x_\beta:\beta<\alpha\}}=n_\alpha! S\setminus H'_\alpha\not=\emptyset,
$
and so there exists  $x_\alpha\in S$  such that   $n_\alpha! x_\alpha\not\in H'_\alpha$. In particular, 
\begin{equation}
\label{small:powers}
nx_\alpha\not\in H'_\alpha
\mbox{ for all }
n\le n_\alpha.
\end{equation}

Let $\kappa=2^\tau$ and $K=\varphi'_\alpha(H'_\alpha)$. For $\xi\in 2^\tau$, define $V_\xi=v_\alpha(U_{\xi,\mathscr{U}_\alpha})$.  Then \eqref{n:alpha} yields
\begin{equation}
\label{W:xi}
\fraction{2}{n_\alpha}<l(v_\alpha(U_{\xi,\mathscr{U}_\alpha}))=l(V_\xi)
\ \ 
\mbox{ for every }
\ \ 
\xi\in2^\tau.
\end{equation}

We need to consider two cases.  

\smallskip

{\sl Case 1\/}. $\{n\in\N\setminus\{0\}: nx_\alpha\in H'_\alpha\}=\emptyset$. In this case, $\grp{x_\alpha}\cong\Z$ and the sum  $\grp{x_\alpha} +H'_\alpha = \grp{x_\alpha} \oplus H'_\alpha$ is direct. Since $\{l(V_\gamma):\gamma<\kappa\}$ has a positive lower bound by \eqref{W:xi}, we can apply Lemma~\ref{second:lemma} to choose 
\begin{equation}
\label{choose:f}
f\in \prod\left\{V_\xi:\xi\in 2^\tau\right\}=F(\mathscr{U}_\alpha, v_\alpha)
\end{equation}
with $\grp{f}\cong\Z$ and $\grp{f}\cap K=\{0\}$. Then the sum $K+\grp{f}=K\oplus\grp{f}$ is direct as well. Since $\grp{x_\alpha}\cong\grp{f}\cong\Z$, there exists a unique monomorphism $\varphi_\alpha: H_\alpha=H'_\alpha\oplus \grp{x_\alpha}\to K\oplus\grp{f}\subseteq \T^{2^\tau}$ extending $\varphi'_\alpha$ such that  $\varphi_\alpha(x_\alpha)=f$.

\smallskip

{\sl Case 2\/}.  $\{n\in\N: nx_\alpha\in H'_\alpha\}\not=\emptyset$.  Let  $m=\min\{n\in \N:  nx_\alpha\in H'_\alpha\}$ and $f'=\varphi'_\alpha(m x_\alpha)\in K$.  Then $m>n_\alpha$ by (\ref{small:powers}), so from \eqref{W:xi} we obtain  $\fraction{2}{m}<\fraction{2}{n_\alpha}<l(V_\xi)$ for every $\xi\in2^\tau$. Since $n_\alpha\ge 1$, we have $m\ge 2$. Applying Lemma~\ref{first:lemma},  we get $f$ satisfying \eqref{choose:f} such that $m f=f'$ and $nf\not\in K$ for all $n\in\N$ with $1\le n<m$. Observe that $\T^{2^\tau}$ (taken as $G^*$), 
$H_\alpha'$ (taken as $K$), $K$ (taken as $K^*$), $x_\alpha$ (taken as $x$), $f$ (taken as $x^*$),  $\varphi'_\alpha$ (taken as $\psi$)
and $m$ satisfy the assumptions of Lemma~\ref{Lemma:extension}. Denoting $\varphi$ from the conclusion of this lemma by $\varphi_\alpha$, we obtain a monomorphism $\varphi_\alpha: H_\alpha\to \T^{2^\tau}$ extending $\varphi'_\alpha$ such that $\varphi_\alpha(x_\alpha)=f$. 

\smallskip

 The monomorphism $\varphi_\alpha$ obviously satisfies (i$_\alpha$), (ii$_\alpha$) and (iii$_\alpha$) in both cases.

\smallskip

The recursive construction being complete, let 
$$
H=\bigcup_{\alpha<\tau} H_\alpha
\ \ 
\mbox{ and }
\ \ 
\varphi=\bigcup_{\alpha<\tau}\varphi_\alpha.
$$
Since (ii$_\alpha$) and (iii$_\alpha$) are satisfied for every $\alpha<\tau$, $\varphi:H\to \T^{2^\tau}$ is a monomorphism.

It remains only to check that $\varphi(H\cap S)$ is dense in $\T^{2^\tau}$. Let $W$ be a non-empty open subset of $\T^{2^\tau}$. Choose pairwise distinct $\xi_1,\dots,\xi_i\in 2^\tau$ and $V_1,\dots,V_i\in\mathscr{V}$ such that 
$$
\left\{f\in \T^{2^\tau}: f(\xi_j)\in V_j
\mbox{ for all }
j\le i\right\}
\subseteq W.
$$ 
Select
$\mathscr{U}=\{U_1,\dots,U_i\} \in\mathbb{U}$ which separates $\xi_j$'s; that is, $j,k\le i$ and $\xi_j\in U_k$ implies $j=k$. Define $v:\mathscr{U}\to \mathscr{V}$ by $v(U_j)=V_j$ for $j\le i$.  Then $(\mathscr{U},v)\in \mathbb{E}$, and so $(\mathscr{U},v)=(\mathscr{U}_\alpha,v_\alpha)$ for some $\alpha<\tau$.  Note that 
\begin{equation}
\label{eq:F:subset:of:V}
F(\mathscr{U}_\alpha, v_\alpha)= F(\mathscr{U},v)\subseteq  \left\{f\in \T^{2^\tau}: f(\xi_j)\in V_j
\mbox{ for all }
j\le i\right\} \subseteq W.
\end{equation}
Since $x_\alpha\in S\cap H_\alpha\subseteq H\cap S$ and 
$
\varphi(x_\alpha)=\varphi_\alpha(x_\alpha)\in F(\mathscr{U}_\alpha, v_\alpha)
$
by (i$_\alpha$),  it follows from (\ref{eq:F:subset:of:V}) that $\varphi(x_\alpha)\in \varphi(H\cap S)\cap W\not=\emptyset$.

\section{Questions}

All instances of potential density in Section~\ref{potential:density:section} were witnessed by a precompact Hausdorff group topology. This makes it natural to ask the following

\begin{question}
Let $S$ be a potentially dense subset of an abelian group $G$. Does there exist a Hausdorff precompact group topology $\mathscr{T}$ on $G$ such that $S$ is $\mathscr{T}$-dense in $G$? 
\end{question}

The answer to this question is positive when $S$ is countable~\cite{DS_Kronecker}. 

It was shown in~\cite{DS_Abelian_MZ} that every potentially dense subset $S$ of an abelian group $G$ must be Zariski dense in $G$.
(We refer the reader to the Introduction for the definition of Zariski density.) Corollary~\ref{necessary:condition:corollary} gives another necessary condition for potential density of $S$ in $G$. One may wonder if these two conditions, combined together,  are also sufficient:

\begin{question}
Let $S$ be an infinite subset of an Abelian group $G$  such that:
\begin{itemize}
\item[(a)] $S$ is Zariski dense in $G$, and 
\item[(b)]  
$\log \log |nG| \leq |nS|$ \ 
(equivalently, $|nG|\le 2^{2^{|nS|}}$)  for all $n\in\N$.
\end{itemize}
Is $S$ potentially dense in $G$? Does there exist a Hausdorff precompact group topology $\mathscr{T}$ on $G$ such that $S$ is $\mathscr{T}$-dense in $G$? 
\end{question}

We note that this question is an appropriate modification of~\cite[Question 45]{DS_OPIT} 
that is necessary in view of Example~\ref{potentially:dense:example}.

\section*{Acknowledgment}  It is our pleasure to thank the referee for her/his
useful suggestions that inspired the authors to re-design Section 3.

\end{document}